\documentclass[12pt]{article} 

\usepackage{times}
\usepackage{latexsym}
\usepackage{amsmath}
\usepackage{amsfonts}
\usepackage{xcolor}
\usepackage{tikz}
\usepackage{amsthm}
\usepackage{parskip}
\usepackage{enumitem}

\begingroup
    \makeatletter
    \@for\theoremstyle:=definition,remark,plain\do{%
        \expandafter\g@addto@macro\csname th@\theoremstyle\endcsname{%
            \addtolength\thm@preskip\parskip
            }%
        }
\endgroup
\makeatletter
\newtheorem*{rep@theorem}{\rep@title}
\newcommand{\newreptheorem}[2]{%
\newenvironment{rep#1}[1]{%
 \def\rep@title{#2 \ref{##1}}%
 \begin{rep@theorem}}%
 {\end{rep@theorem}}}
\makeatother

\usepackage{epstopdf}
\usepackage{hyperref}
\newtheorem{Definition}[equation]{Definition}
\newcounter{mark}

\newtheorem{Walk}[mark]{Random Walk}

\newtheorem{Theorem}[equation]{Theorem}
\newtheorem{Proposition}[equation]{Proposition}
\newtheorem{DProposition}[equation]{Dodgy Proposition}
\newtheorem{Lemma}[equation]{Lemma}
\newtheorem{Corollary}[equation]{Corollary}
\newtheorem{Remark}[equation]{Remark}
\newtheorem{Problem}[equation]{Problem}

\newtheorem{Question}[equation]{Question}

\theoremstyle{definition}
\newtheorem{Example}[equation]{Example}
\definecolor{darkred}{rgb}{0.7,0,0} 
\newcommand{\un}{\pi}
\DeclareMathOperator{\im}{im}

\newreptheorem{Theorem}{Theorem}
\newreptheorem{Mark}{Marking Scheme}

\newcommand{\M}{\mathcal{M}}

\newcommand{\figredpointsmodsix}{\begin{tikzpicture}
\draw (0,0) circle (2cm);
\filldraw[red] (0:2) circle (0.05cm) (20:2) circle (0.05cm) (60:2) circle (0.05cm) (120:2) circle (0.05cm) (140:2) circle (0.05cm) (180:2) circle (0.05cm) (240:2) circle (0.05cm) (260:2) circle (0.05cm) (300:2) circle (0.05cm);
\filldraw[blue] (40:2) circle (0.05cm) (80:2) circle (0.05cm) (100:2) circle (0.05cm) (160:2) circle (0.05cm) (200:2) circle (0.05cm) (220:2) circle (0.05cm) (280:2) circle (0.05cm) (320:2) circle (0.05cm) (340:2) circle (0.05cm);
\foreach \i in {0,...,17}
{
\draw (\i*20:1.75) node{$\i$};
}
\begin{scope}[xshift=5cm]
\draw (0,0) circle (2cm);
\filldraw[red] (0:2) circle (0.05cm) (60:2) circle (0.05cm) (120:2) circle (0.05cm) (180:2) circle (0.05cm) (240:2) circle (0.05cm) (300:2) circle (0.05cm);
\filldraw[blue] (40:2) circle (0.05cm) (100:2) circle (0.05cm) (160:2) circle (0.05cm) (220:2) circle (0.05cm) (280:2) circle (0.05cm) (340:2) circle (0.05cm);
\end{scope}
\begin{scope}[xshift=10cm]
\draw (0,0) circle (2cm);
\filldraw[red] (20:2) circle (0.05cm) (140:2) circle (0.05cm) (260:2) circle (0.05cm);
\filldraw[blue] (80:2) circle (0.05cm) (200:2) circle (0.05cm) (320:2) circle (0.05cm);
\end{scope}
\draw (2.5,0) node{$=$} (7.5,0) node{$+$};
\end{tikzpicture}}

\title{Strong stationary times for features of random walks}
\author{Graham White}
\date{\today}
\begin{document}

\maketitle

\begin{abstract}
In \cite{GWcouplingfeatures}, we examined the use of coupling to obtain bounds on the mixing time of statistics on Markov chains. In the present paper, we consider the same general problem, but using strong stationary times rather than coupling. We discuss various types of behaviour that may occur when this is attempted, and analyse a variety of examples. 
\end{abstract}

\section{Introduction}
\label{sec:intro}

This paper is a sequel to \cite{GWcouplingfeatures}. As in that paper, we are interested in the following general problem:

\begin{Problem}
If $\M$ is a Markov chain, and $f$ is a function defined on the states of $\M$, how long must $\M$ be run to guarantee that the distribution of $f$ is close to what it would be on the stationary distribution of $\M$?
\end{Problem}

In \cite{GWcouplingfeatures}, we explored what the technique of coupling could say about this problem, while in this paper we will use strong stationary times. See that paper for a discussion of related work.

Strong stationary bound techniques give bounds on separation distance, so we will use this distance throughout the paper. As with the the mixing of actual Markov chains, upper bounds on the separation distance of a statistic away from its stationary distribution are stronger than the same bounds in total variation distance.

\subsection*{Acknowledgements}

I am grateful to my advisor, Persi Diaconis, for consistent helpful advice, illuminating conversations, and good cheer.

\section{Statistics on chains}
\label{sec:stats}

We will start by describing how a strong stationary time may be used to obtain a bound on the mixing time of a statistic.

\begin{Proposition}
\label{prop:sdcoupling2}
Let $\M$ be a Markov chain on a finite state space $\Omega$ with stationary distribution $\un$. Let $p$ be between 0 and 1, $t$ a positive integer, and $f$ a function on $\Omega$. If for any initial state $x_0$ and any $a \in \im(f)$, there is at least a probability $(f(\un))(a)p$ that $f(X_t) = a$, then for any initial state $x_0$ the distribution $f(X_t)$ is within $(1-p)$ of the stationary distribution $f(\un)$ in separation distance.
\end{Proposition} 
\begin{proof}
Consider $X_t$. For each $a \in \im(f)$, there is a probability of at least $(f(\un))(a)p$ that $f(X_t) = a$, so the separation distance of $f(X_t)$ from the stationary distribution $f(\un)$ is at most $(1-p)$.
\end{proof}

This proof is essentially the same as the standard proof in the more familiar setting of Markov chain mixing.

For our initial set of examples, we will consider two shuffling schemes (that is, random walks on the symmetric group $S_n$) for which strong stationary times are available. The first is the random-to-top shuffle, where at each step a random card is chosen and moved to the top of the deck. For this process, the time taken to choose every card is a strong stationary time. The second shuffling scheme we will consider is inverse riffle shuffles, where at each step, independent uniform random bits are assigned to each card, and then the deck is sorted by these bits, with ties broken by the initial order of those cards. Multiple steps of this process may be implemented by assigning uniform longer binary strings, and sorting by those. For this process, a strong stationary time is the time taken until every card has been assigned a different string from every other card.

We will examine what these strong stationary times have to say about various statistics on the symmetric group, such as the identity of the top card, or top $k$ cards, the location of any specific card or set of cards, the set of cards in each quarter of the deck, or in each hand when the cards are dealt one at a time, the parity of the permutation, the identity of the card above or below any specific card, the relative order of any set of cards, or the distance between two fixed cards.  

\subsection{A cautionary tale}

Strong stationary times are notoriously delicate, and it is easy to make mistakes. Often, these happen when one states an intuitive result like `This set of cards are equally likely to be in any order' and is insufficiently careful with the conditions and qualifiers surrounding that statement. We will give examples where Proposition \ref{prop:sdcoupling2} may be used to bound the separation distance mixing time in Sections \ref{sec:rttssts} and \ref{sec:invriffssts}. Before doing that though, the following example shows a subtle way in which these examples can go awry, illustrating why care is necessary. The issue is similar to that discussed in Remark 11 of \cite{GWcouplingfeatures}.

It is tempting to immediately brandish Proposition \ref{prop:sdcoupling2} at the various statistics listed in the previous section. For example, one might attempt arguments like the following:

\begin{DProposition}
\label{dprop:rtttoptwo1}
In the random-to-top shuffle, the top two cards are random after two different cards have been chosen.
\end{DProposition}
\begin{proof}[`Proof':]
Consider a sequence of moves in which exactly two different cards $a$ and $b$ are chosen. This sequence could be modified by instead moving card $c$ instead of card $a$ and card $d$ instead of card $b$ at each time either of those cards would be moved. This is a bijection to sequences of moves in which exactly two cards are moved and which end with cards $c$ and $d$ on the top of the deck, rather than cards $a$ and $b$.
\end{proof}

\begin{DProposition}
\label{dprop:rtttoptwo2}
The separation distance between the distribution of the top two cards after $t>1$ steps and the uniform distribution is at most $\frac{1}{n^{t-1}}$.
\end{DProposition}
\begin{proof}[`Proof':]
The probability that two different cards have been chosen after $t > 1$ steps is $1 - \frac{1}{n^{t-1}}$. By Proposition \ref{prop:sdcoupling2} and (dodgy) Proposition \ref{dprop:rtttoptwo1}, the result follows.
\end{proof}

The conclusions of Propositions \ref{dprop:rtttoptwo1} and \ref{dprop:rtttoptwo2} are correct, but their methods are subtly flawed. The error lies in the interaction between the two --- Proposition \ref{dprop:rtttoptwo1} is not specific enough about what it claims, and the statement proved is different to that used in the proof of Proposition \ref{dprop:rtttoptwo2}. The proof of Proposition \ref{dprop:rtttoptwo1} only discusses paths which move exactly two different cards, while the proof of Proposition \ref{dprop:rtttoptwo2} uses a stronger version of Proposition \ref{dprop:rtttoptwo1} which is true for all paths involving at least two different cards.

In this case, these errors are easily fixable, and this will be done in the proof of Proposition \ref{prop:rttssttop2}. The following is an example where an error of this type will give a result that is far too strong.

Consider the following shuffling scheme.

\begin{Walk}
\label{wal:rtft}
Take a deck of $n$ cards, and at each step, with probability $\frac{1}{2}$ either move a random card to the top of the deck, chosen uniformly, or with probability $\frac{1}{2}$ move the top card of the deck to the bottom of the deck.
\end{Walk}

Consider shuffling a deck in this way, while being interested in the value of the top card. Making the same mistakes as in Propositions \ref{dprop:rtttoptwo1} and \ref{dprop:rtttoptwo2} in the analysis of this statistic leads to results that are far too strong.

\begin{DProposition}
\label{dprop:rtfttop1}
In Random Walk \ref{wal:rtft}, the top card is random after any card has been moved to the top of the deck.
\end{DProposition}
\begin{proof}[`Proof':]
If a card $a$ is moved to the top of the deck, then it was equally likely that any other card $b$ was chosen and moved to the top of the deck in that step. Thus, moving $b$ to the top instead is a bijection between sequences of moves ending with moving card $a$ to the top and sequences of moves ending with moving card $b$ to the top, showing that it is equally likely that the top card is equal to $a$ as to $b$.
\end{proof}

\begin{DProposition}
\label{dprop:rtfttop2}
The separation distance between the distribution of the top card after $t$ steps and the uniform distribution is at most $\frac{1}{2^{t}}$.
\end{DProposition}
\begin{proof}[`Proof':]
After $t$ steps, the probability that a card has been moved to the top in one of those steps is $1 - \frac{1}{2^{t}}$. By Proposition \ref{prop:sdcoupling2} and (dodgy) Proposition \ref{dprop:rtttoptwo1}, the result follows.
\end{proof}

The conclusion of Proposition \ref{dprop:rtfttop2} is grossly incorrect. To see this, take the concrete example of $n=52$ and $t=10$. Proposition \ref{dprop:rtfttop2} states that the distribution of the top card after $10$ steps is within $\frac{1}{2^{10}}$ of uniform. In particular, the probability that the top card is the card initially on the bottom of the deck should be at least $\frac{2^{10}-1}{2^{10}}\cdot\frac{1}{52}$. 

There are two possibilities regarding the eventual top card --- it must have either been moved to the top at some point, or not. If it has been moved to the top, then by moving a different card to the top and then proceeding with the rest of the path, that card will end up on top. Hence, among paths where the eventual top card was at some point moved to the top, the identity of that card is uniformly distributed.

However, if the eventual top card was never moved to the top, then it must have already been somewhere close to the top --- for $t=10$, only cards $1$ through $11$ are possibilities. 

There is a probability of $\frac{252}{1024}$ that the sequence of moves is such that every time a card is moved to the top, it is moved to the bottom at some future point --- this can be seen by reversing the path and observing that of the $2^{10}$ paths of length $10$ comprised of steps up-right or down-right, $252$ of them never move below the $x$--axis. In such cases, the eventual top card was never moved to the top.

Thus the probability that the top card after $10$ steps is the $52$ is at most $\frac{1024-252}{1024}\cdot\frac{1}{52}$ (in fact, this probability is exact), so the separation distance of the distribution after $10$ steps is at least $\frac{252}{1024} \approx \frac{1}{4}$, which is much further than the $\frac{1}{1024}$ promised by the author of Proposition \ref{dprop:rtfttop2}.

As discussed following Propositions \ref{dprop:rtttoptwo1} and \ref{dprop:rtttoptwo2}, the errors are that Proposition \ref{dprop:rtfttop1} is not specific enough about what it claims and that Proposition \ref{dprop:rtfttop2} uses a stronger version of Proposition \ref{dprop:rtfttop1} than was actually proven.

Proposition \ref{dprop:rtfttop1} is true if it is taken to only apply to paths ending in a random-to-top move. It is also true in the more general case where the eventual top card was at some point moved to the top. As in the preceding discussion, this proposition fails in the case where the eventual top card was never moved to the top. 

The supposed proof of Proposition \ref{dprop:rtfttop2} applies Proposition \ref{dprop:rtfttop1} to the wider class of paths which at any point move a card to the top. That is, the problem is that it is possible to move from a `good' state, where Proposition \ref{dprop:rtfttop1} applies and the distribution of the top card is uniform, to a `bad' state, where Proposition \ref{dprop:rtfttop1} no longer applies. 

This is not behaviour that can appear when strong stationary times are used for the convergence of a Markov chain, because once a chain is in the stationary distribution, it remains there. This is not true for a statistic on a Markov chain.

Strong stationary times are fragile --- it is important to be clear about exactly which paths are under consideration. 

\subsection{Random to top shuffle}
\label{sec:rttssts}

We now examine whether there are easy-to-spot strong stationary times for our statistics. Many of these times are the same as the coupling times for the same statistic in Section 3.1 of \cite{GWcouplingfeatures}. When these coupling times are strong stationary times in the sense of Proposition \ref{prop:sdcoupling2}, we obtain a bound on the mixing time of that statistic. The second of these examples will include corrected proofs of Propositions \ref{dprop:rtttoptwo1} and \ref{dprop:rtttoptwo2}.

\begin{enumerate}
\item The top card of the deck is uniformly distributed after one step, as each card has probability $\frac{1}{n}$ of being the card moved to the top. A bijective proof of this statement is to consider any path and to change the choice of the last card moved to the top.
\item The coupling condition for the identities of the top two cards was for two different cards to have been chosen. At this time, the identities of those cards are uniformly distributed. The following proof provides a more careful (and correct!) version of the `proof' of Proposition \ref{dprop:rtttoptwo2}.

\begin{Proposition}
\label{prop:rttssttop2}
At least two different cards having been chosen is a strong stationary time for the statistic of which two cards are in the top two positions.
\end{Proposition}
\begin{proof}
More precisely, the claim is that among paths which choose at least two different labels, each of the $n(n-1)$ potential ordered pairs of values of the top two cards appear equally many times. 

If a path chooses at least two different labels, then the card which ends up as the top card of the deck is the last one chosen, and the card which ends up as the second-to-top card is the second-to-last one chosen. For any $a,b,c,d$ between $1$ and $n$ with $a \neq b$ and $c \neq d$, it will be equally likely that the top two cards are $(a,b)$ as that they are $(c,d)$, given that at least two different cards have been chosen.

This will be shown by constructing a bijection between paths ending with the top cards being $(a,b)$ and paths ending at $(c,d)$. Let $m_i$ denote moving the card of label $i$ to the top. Let $g$ be an arbitrary element of $S_n$ with $g(a) = c$ and $g(b) = d$.

Consider a path $p$ which ends with $(a,b)$ on top of the deck, \[p = m_{i_1}m_{i_2}\dots m_{i_t}.\] Define $g(p)$ to be \[g(p) = m_{g(i_1)}m_{g(i_2)}\dots m_{g(i_t)}.\] (Here the permutation applied to the indices is $g$ rather than $g^{-1}$). The path $g(p)$ ends with $g(a) = c$ and $g(b) = d$ on top of the deck, and this map is invertible, completing the proof.
\end{proof}

This proof is more complicated than that required for the top card alone, which only required modifying the final step of the path. One might be tempted to try something similar to that proof --- changing all the $a$'s and $b$'s to $c$'s and $d$'s respectively, but not modifying other labels. While this would produce a path that resulted in the correct top two cards, this map is not a bijection --- for example it sends both $m_cm_bm_a$ and $m_am_bm_a$ to $m_cm_dm_c$. It is possible to do something similar by considering the last time each card is moved.

Having shown that choosing two different cards is a strong stationary time, the distribution of this time is a sum of two independent geometric random variables.

\item The top $k$ cards are random after $k$ different cards have been chosen, and at this time this statistic is uniformly distributed. This is proven in the same way as Proposition \ref{prop:rttssttop2}, considering paths that include $k$ different cards rather than two, and using a permutation $g$ which takes the $k$ most recently moved cards to any other ordered $k$--tuple.
\item The location of the $1$ is not uniformly distributed when it couples --- matches are only created at the top of the deck.
\item Likewise, the locations of other specific cards are not uniformly distributed when they couple.
\item The sets of cards in each quarter of the deck are uniformly distributed once three quarters of the labels have been chosen. As with the coupling for this statistic, the result is implied by the stronger property that the labels and order of the top three quarters of the deck are known. If that statistic is uniformly distributed, then so is this one.
\item For the sets of cards in positions congruent to each $i$ modulo $4$, the coupling time was of the same order as the coupling time for the state of the deck itself. After this time, the state of the deck is uniformly distributed, which implies that this statistic is also uniformly distributed.
\item For the parity of the permutation, let $m_i$ denote moving the $i$th card to the top of the deck. Notice that $m_i$ is an even permutation whenever $i$ is odd, and vice versa. Let $p$ be an arbitrary sequence of random-to-top moves \[p = m_{a_1}m_{a_2}\dots m_{a_k}.\]

If $n$ is even, then modify $p$ by increasing $a_k$ by one if $a_k$ is odd and decreasing it by one if it was even. This is a bijection between paths producing odd permutations and paths producing even permutations, so the parity of the permutation mixes perfectly after a single step.

If $n$ is odd, then let $j$ be the greatest index so that $a_j$ is not equal to $n$. As before, modify $p$ by increasing $a_j$ by one if $a_j$ is odd and decreasing it by one if it was even. Restricting to paths where not all $a_i$ are equal to $n$, this is a bijection between paths producing odd permutations and paths producing even permutations. Therefore the separation distance between the distribution of the parity of the permutation after $k$ steps and the uniform distribution is at most $\frac{1}{n^k}$. This bound is exact, because only exactly one path of each length was omitted from consideration.

\item The identity of the card immediately above the $1$ is not uniformly distributed as soon as the $1$ has been chosen, because it is more likely than average that the $1$ is on top of the deck. The other values of this statistic, $2$ through $n$, are equally likely as long as the $1$ has been chosen, though. A bijection to show this is as follows. Consider a path in which the card above the $1$ is the $k$. Modify the path by replacing each choice of the $k$ with the $l$ and each choice of the $l$ with the $k$. This is a bijection between paths ending with the $k$ immediately above the $1$ and paths ending with the $l$ immediately above the $1$.  

While this means that this stopping time is not a strong stationary time, it does give quite a lot of information about the distribution of the statistic after this time --- namely, that of its $n$ possible values, $(n-1)$ of them are equally likely.

\item Likewise, the identity of the card immediately below the $1$ is mostly uniformly distributed (among $2$ through $n$, not including the possibility that the $1$ is on the bottom of the deck) given that the $1$ has been chosen and that if not all cards have been chosen, then it has been chosen more recently than at least one other card. This condition does not define a stopping time, because whether or not it is true can change several times --- it is just a statement about which set of paths the appropriate bijection is defined on.
\end{enumerate}

\begin{Remark}
This example is another illustration of what may go wrong when the necessary conditions for the statistic to be understood may change from being satisfied to not. In contrast, if working with the convergence of a Markov chain, then once the chain were in its stationary distribution, it would remain there. 
\end{Remark}

\begin{enumerate}[resume]
\item A similar statement is true for the identities of the $k$ cards after the $1$ --- as long as the $1$ has been chosen more recently than at least $k$ other cards, then any value of this statistic where the $1$ has at least $k$ cards below it is equally likely. 
\item The relative order of the $1$ and $2$ is uniformly distributed as soon as either card has been chosen --- on paths where at least one of those cards has been chosen, replacing every choice of the $1$ by the $2$ and vice versa is a bijection between paths ending with the $1$ above the $2$ and the reverse.
\item The relative order of the $1$ to $k$ is uniformly distributed as soon as $k-1$ of those cards have been chosen. On such paths, replacing every choice of label $i$ with label $g(i)$ is a bijection between paths with the $1$ to $k$ ending up in relative order $h$ and paths with the $1$ to $k$ ending up in relative order $g^{-1}(h)$, for any permutation $g$ of the labels $\{1, 2, \dots, k\}$.
\item The number of cards between the $1$ and the $2$ is not distributed as its stationary distribution (which is not uniform) when both of those labels have been chosen --- it is more likely to be small, given that both of those cards have been moved to the top of the deck.
\end{enumerate}

In the above examples, it appears that statistics involving the values of cards are likely to be uniformly distributed, while statistics involving the positions of cards are not. This seems reasonable, given that the random-to-top walk treats all card values in the same way, but treats positions differently. The (rather unnatural) dual walk, where a card is chosen at random, its label changed to $1$, and all smaller labels increased by $1$, would have the reverse property.

\subsection{Inverse riffle shuffles}
\label{sec:invriffssts}

As with the previous section for the random-to-top process, this section will explore strong stationary times for inverse riffle shuffles. We will discuss which of our statistics are susceptible to the present techniques, excluding those for which the coupling time in Section 3.2 of \cite{GWcouplingfeatures} was no better than the coupling time for the entire deck. Recall that inverse riffle shuffles are implemented by assigning independent binary strings to each card and then sorting the deck by these strings. For the random-to-top shuffle, the bijections on paths in the previous section involved changing which cards were moved to the top at various times. Similarly, for the inverse riffle shuffle, similar bijections will change which strings are assigned to which card. One could also modify which strings were assigned, which will not be necessary in these examples.

The times taken until these various conditions are satisfied are discussed in Section 3.2 of \cite{GWcouplingfeatures}.

\begin{enumerate}
\item The identity of the top card is uniformly distributed once the first (lexicographically) string is different from all others. Given any (multi-)set of strings that have been assigned, all possible assignments of those strings to the various cards are equally likely, so which card has been assigned the unique smallest string is equally likely to have been any of them.

To write down a bijection between paths resulting in card $k$ having the unique smallest string and paths resulting in card $l$ having the unique smallest string, choose an arbitrary permutation $g$ that takes $k$ to $l$, and for each string that would be assigned to card $i$, assign it to the card $g^{-1}(i)$ instead.
\item Likewise, the identity of the second-to-top card is uniformly distributed once the second string is different from all others, and the identity of the $k$th from top card is uniformly distributed once the $k$th string is different from all others, because again, any set of strings are equally likely to be assigned to the cards in any order.
\item The set of the top $k$ cards is uniformly distributed once the $k$th and $(k+1)$th strings are different, for the same reason.
\item The identity and order of the top $k$ cards is uniformly distributed once the top $(k+1)$ strings are all different.
\item The location of the $1$ is not uniformly distributed once that card is assigned a string distinct from all others, because whether or not a string is unique is not independent of its position. For example, assigning single bits to a deck of three cards can only result in a unique value in the first or third position, never in the second. 

One may deduce weaker information about the distribution, though --- for example, once the $1$ has been assigned a string different from all others, its average position is $\frac{n+1}{2}$, because the configuration obtained by for every string, assigning its opposite instead (swapping zeroes for ones), is equally likely and ends with the $1$ in the mirrored position. This still requires that the $1$ be assigned a string distinct from all others, because this bijection only reverses the order of the blocks of cards corresponding to each string relative to one another, and not the order of the cards within each block.
\item The locations of larger sets of cards are not uniformly distributed when they are assigned unique strings.
\item The bridge hands dealt in blocks are uniformly distributed once the blocks of $13$ consecutive strings do not overlap with one another because again, this is a condition that depends solely on which cards are assigned which strings, and for any set of strings, it is equally likely that they were assigned to the cards in any arrangement.
\item After the $1$ and the card immediately after it are both assigned unique strings, the value of the card after the $1$ is equally likely to be any card, but this is not necessarily the same probability as the chance that the $1$ is at the bottom of the deck (this is the same behaviour as seen in the random-to-top shuffle). 

A bijection for this result would leave the set of strings fixed, keep the same string assigned to the $1$, and change which card receives the next largest string, which is known to be unique.
\item The relative order of the $1$ and the $2$ is uniformly distributed once these two cards are assigned different strings, because again, the strings assigned to the $1$ and $2$ can just be switched.
\item The relative order of the $1$ to the $k$ is uniformly distributed once the strings assigned to those $k$ cards are are all distinct.
\end{enumerate}

As with the random-to-top shuffle, the inverse riffle shuffle treats all values equally but treats positions differently, so it seems reasonable that statistics involving the values of cards are likely to be uniformly distributed, but statistics involving the positions of cards are not. Indeed, the inverse riffle shuffle can be thought of as choosing a random subset of the deck and bringing that to the top.

The preceding results regarding inverse riffle shuffles may be translated to results about forwards riffle shuffles. For example, the time required for the location of the $1$ to be uniformly distributed when the deck is shuffled by forwards riffle shuffles is the same as the time required for the identity of the top card to be uniformly distributed under inverse riffle shuffles, which was discussed above.

\section{Marked points on the cycle}
\label{sec:redpoints}

We will now show how strong stationary times may be used for a different flavour of example. Consider the following setting:

\begin{Example}
\label{ex:redpoints}
Take the $2n$--cycle, with vertices labelled in order by $1,2,\dots,2n$. Half of the vertices are coloured red, the rest blue. A  single particle moves on this graph. At each step, it moves either to the right or left with probability $\frac{1}{4}$, and remains in place otherwise. How long does it take until the probability of the particle being on a red vertex is close to one half, depending on what the set of red vertices is? For example, what happens when the red vertices are:
\begin{itemize}
\item $\{2,4,6,\dots,2n\}$,
\item $\{1,2,3,\dots,n\}$,
\item $\{i, i \equiv 0,1 \text{ or } 3 \mod 6\}$, or
\item A general set of $n$ vertices
\end{itemize}

After $O(n^2)$ steps, the position of the particle is known to be close to random (see for instance Section 3C of \cite{diaconis1988group}), and so the colour is as well. Depending on the arrangement of the red points, the colour may converge to uniform much faster. In the first case, where every second point is red, one step suffices. After a single step from any starting point, the probability of being on a red vertex is exactly $\frac{1}{2}$.

In the second case, a particle starting at $\frac{n}{2}$ takes on average $\frac{n^2}{2}$ steps before even leaving the red region for the first time (Proposition 2.2 of \cite{LPW}). This is comparable to the $O(n^2)$ steps required for the particle's actual position to mix. In the third case, the colour is close to random after $9$ steps, upper bounds for the general case are given in Theorem \ref{the:redpoints}.
\end{Example}

It will be convenient to consider this walk as two consecutive steps of the walk which steps from a vertex to a random adjacent edge, and from an edge to a random adjacent vertex. That is, when the particle stays in place, it is considered to have moved to one of the adjacent edges and then back to the previous vertex.

\subsection{Regular examples}

Before solving the problem in general, it will be instructive to consider some simple examples for choices of the red vertices.

\begin{Example}
The simplest example is when every second vertex is coloured red. In this case, every edge is adjacent to a red vertex and a blue vertex, so whenever the walk moves from an edge to a vertex, it is equally likely that it moves to a red vertex as to a blue one. Hence for any path there is another path which ends on a vertex of the opposite colour, obtained by reversing the direction of the final step. 

This is a bijection between paths of any length which end on red vertices and paths of the same length which end on odd vertices, so there are the same number of paths of that length which end on red vertices and on blue vertices. Hence after a single step, the colour is uniformly distributed.
\end{Example}

This is the key idea for the more general proof --- that for any pair of vertices, once a path passes through their midpoint (on either arc), it is equally likely to end on either of those two vertices. The remainder of the analysis will consist of pairing up close pairs of red and blue points in such a way that most paths of a suitable length pass through one such midpoint, and so have a corresponding path ending on a point on the opposite colour.

\begin{Definition}
A set of points is \emph{alternating} if it consists of alternating red and blue points.
\end{Definition}

Note that an alternating set must contain an even number of points, otherwise there will be two consecutive points of the same colour.

\begin{Definition}
The \emph{midpoints} $M_A$ of an alternating set $A$ of points are the vertices or edges halfway between each consecutive pair of points from $A$. If $A$ has more than two points, then any consecutive pair of points defines only one midpoint, on the arc containing no other points from $A$. If $A$ contains exactly two points, then there is a midpoint on each arc.  
\end{Definition}

Consider the following slightly more complicated example before proceeding to the general case.

\begin{Example}
\label{ex:redpointsmod6}
Consider a cycle with $6n$ vertices, with the red vertices being those congruent to $0,1$ or $3$ mod $6$. Split this problem into multiple instances of the previous example. Let $A$ be the set of vertices congruent to $0,2,3$ or $5 \mod 6$, and $B$ be the set of vertices congruent to $1$ or $4 \mod 6$, as shown in Figure \ref{fig:redpointsmod6}. Note in particular that $A$ and $B$ both consist of alternating red and blue vertices. Let $M_A$ and $M_B$ be the sets of midpoints between consecutive points of $A$ and $B$, respectively. These midpoints might be either vertices or midpoints of edges.

\begin{figure}
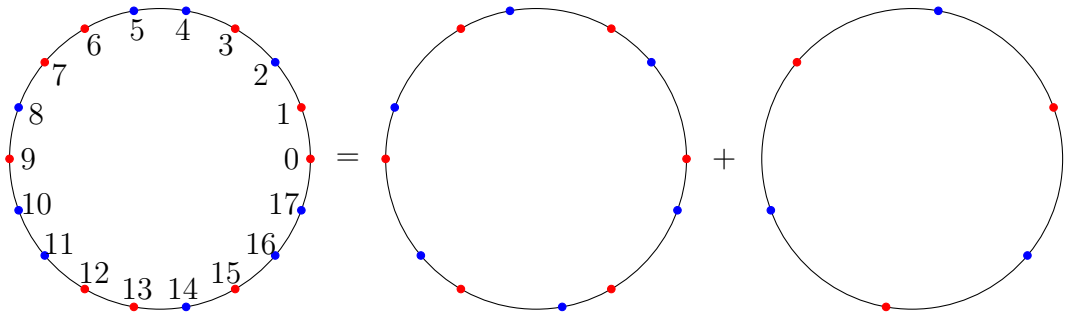

\figredpointsmodsix
\caption[Breaking Example \ref{ex:redpointsmod6} into simpler pieces]{The scenario of Example \ref{ex:redpointsmod6}}
\label{fig:redpointsmod6}
\end{figure}

\begin{Proposition}
\label{prop:redpointsmod6}
Let $T$ be the time taken for a simple random walk on a cycle to move at least three units from its initial position and let $s(t)$ be the separation distance of the colour in Example \ref{ex:redpointsmod6} from uniform after $t$ steps. Then $s(t) \leq \Pr(T > t)$.
\end{Proposition}
\begin{proof}
As in the previous example, it is possible to biject most paths ending on red vertices to paths ending on blue vertices. This scenario is slightly more complicated than the previous example because the distances between consecutive vertices are unequal.

Consider a path which passes through a point in $M_A$ and ends on a point in $A$. Take the last time it crosses a point of $M_A$, and reflect the remaining portion of the path in that point. This produces a path which ends on a point of the opposite colour. Further, this map is an involution. Thus, there are exactly as many paths which cross $M_A$ and end on a blue point of $A$ as there are paths which cross $M_A$ and end on a red point of $A$, and likewise for $M_B$ and $B$.

Thus, a path long enough to cross both $M_A$ and $M_B$ is equally likely to end on a red vertex or a blue vertex. The sets $A$ and $B$ were chosen so that both $M_A$ and $M_B$ are closely spaced, and so a simple random walk on the cycle doesn't take long to cross them. The set $M_A$ is $\{1,2.5,4,5.5\} \mod 6$, and $M_B$ is $\{2.5,5.5\} \mod 6$, where halves indicate that the midpoint in question is on an edge rather than a vertex.

Note that no pair of consecutive points in $M_A$ are further than $1.5$ units from one another, and no pair of consecutive points in $M_B$ are further than $3$ units from one another. Hence any path which has moved at least $3$ units from its initial position has crossed both $M_A$ and $M_B$, and the time taken for this is about $3^2$ steps. 
\end{proof}

That is, the time taken for the colour of the current position to mix is less than the time taken for a simple random walk to move at least three steps from its starting location, and this latter problem is well understood.
\end{Example}

\subsection{The general case}
\label{sec:generalredpoints}

Finally, consider arbitrary sets of red points. The added difficulties compared to the previous instance of the problem are that the points need not be periodic, and that it might be necessary to break the problem down into more than two sets of alternating red and blue vertices. Indeed, in the worst case where the $n$ red points form a contiguous block on one side of the cycle, it is impossible to partition the vertices into any fewer than $n$ sets of alternating red-blue points, each of size only two.

Firstly, consider the minimal number of alternating red-blue sets into which the cycle can be partitioned.

\begin{Lemma}
\label{lem:kalternating1}
Consider an arbitrary arrangement of $n$ red points and $n$ blue points. Start at any point and proceed around the cycle, keeping track of the difference $(R - B)$, where $R$ and $B$ are respectively the numbers of red and blue points visited. Let $k$ be the difference between the maximum and minimum values of $(R-B)$. Then $k$ is the minimum number of alternating red-blue sets into which the cycle can be partitioned. That is, it is possible to partition the cycle into $k$ alternating sets, but not into $k-1$.
\end{Lemma} 
\begin{proof}
Without loss of generality, assume that the quantity $(R-B)$ never fell below zero. If it did, then just start again from the point where $(R-B)$ achieved its minimal (negative) value. This will change $R$ and $B$ by constant amounts, leaving the maximum and minimum values of $(R-B)$ unaffected, and thus not changing $k$. Hence, assume that this process has always visited at least as many red vertices as blue.

It is necessary to show that it is possible to partition the cycle into $k$ alternating sets. Let the red points, in the order they were visited, be $(R_1,R_2,\dots,R_n)$ and the blue points be $(B_1,B_2,\dots,B_n)$. Define $k$ sets of these points as follows. Let \[A_1 = \{R_1,B_1,R_{k+1},B{k+1},R_{2k+1},B_{2k+1},\dots\}\] and more generally, for each $i$ between $1$ and $k$, let \begin{equation}\label{eq:alternating}A_i = \{R_i,B_i,R_{k+i},B{k+i},R_{2k+i},B_{2k+i},\dots\}.\end{equation} These are finite sets which terminate once the indices of the points exceed $n$.

Once it has been checked that these sets are alternating, then they form a decomposition of the cycle into $k$ alternating sets, as claimed.

To verify this, it suffices to show that the vertices appear in the order that they are given in the definition of the sets in Equation \ref{eq:alternating}. That is, that for each $a$ and $i$, the vertex $R_{ak+i}$ occurs before $B_{ak+i}$, which occurs before $R_{(a+1)k+i}$. The first inequality is true because it was assumed that at any point, at least as many red vertices have been encountered as blue vertices, while the second is true because otherwise when $R_{(a+1)k+i}$ is reached, the process will have encountered more than $k$ more red vertices than blue, contradicting the definition of $k$.

Thus this is a decomposition of the cycle into $k$ alternating sets.  

To show that it is impossible to decompose the cycle into fewer than $k$ alternating sets, recall from the definition of $k$ that at some point exactly $k$ more red points than blue points have been encountered. If there are fewer than $k$ alternating sets, then there is at least one of them such that at least two more red points than blue from that set have been covered. But each set is alternating, so this is impossible.
\end{proof}

The next step is to improve Lemma \ref{lem:kalternating1} to also give an upper bound on the distance between any two adjacent points of the alternating sets.

\begin{Lemma}
\label{lem:kalternating2}
Consider an arbitrary arrangement of $n$ red points and $n$ blue points, with $k$ defined as in Lemma \ref{lem:kalternating1}. It is possible to partition the cycle into $k$ alternating sets such that in each of these sets, any two consecutive points are at most distance $2k-1$ apart.
\end{Lemma}
\begin{proof}
It will be checked that this bound is obeyed by the construction presented in the proof of Lemma \ref{lem:kalternating1}, with \[A_i = \{R_i,B_i,R_{k+i},B{k+i},R_{2k+i},B_{2k+i},\dots\}.\]

Consider the vertices that could appear between $R_{ak+i}$ and $B_{ak+i}$. Red vertices with index lower than $ak+i$ appear before $R_{ak+i}$, and those with index at least $(a+1)k+i$ appear after $B_{ak+i}$. Likewise, blue vertices with index at most $(a-1)k+i$ appear before $R_{ak+i}$, and those with index greater than $ak+i$ appear after $B_{ak+i}$. Thus, the only vertices which could appear between $R_{ak+i}$ and $B_{ak+i}$ are $R_{ak+i+1}$ through $R_{(a+1)k+i-1}$ and $B_{(a-1)k+i+1}$ through $B_{ak+i-1}$, a total of $2k-2$ vertices. Thus the distance between $R_{ak+i}$ and $B_{ak+i}$ is at most $2k-1$.

Similarly, the only vertices that could appear between $B_{ak+i}$ and $R_{(a+1)k+i}$ are $B_{ak+i+1}$ through $B_{(a+1)k+i-1}$ and $R_{ak+i+1}$ through $R_{(a+1)k+i-1}$. There are $2n-2$ such vertices, so the distance between $B_{ak+i}$ and $R_{(a+1)k+i}$ is at most $2k-1$.

It may not be obvious where the definition of $k$ was used in this proof --- the sets \[A_i = \{R_i,B_i,R_{k+i},B{k+i},R_{2k+i},B_{2k+i},\dots\}\] are alternating, and this was proved in the proof of Lemma \ref{lem:kalternating2} from the definition of $k$.
\end{proof}

Continuing, this decomposition into $k$ alternating sets gives the following result.

\begin{Theorem}
\label{the:redpoints}
A strong stationary time for the colour of the present vertex to be uniformly distributed is for the random walk to have been to at least $2k-1$ different vertices, with $k$ defined as in Lemma \ref{lem:kalternating1}. 
\end{Theorem}
\begin{proof}
As in Example \ref{ex:redpointsmod6}, if a path passes through a midpoint of one of the alternating sets $A_i$ and ends on a vertex in $A_i$, then there is another path that turned the other way the last time it left a midpoint of $A_i$, ending on a vertex of $A_i$ of the opposite colour. This defines a bijection between such paths ending on a red vertex of $A_i$ and ending on a blue vertex.

Hence, paths that have passed at least one midpoint of each set $A_1,A_2,\dots$ and $A_k$ are equally likely to end on a red vertex as a blue vertex, because whichever vertex they end on belongs to one of these sets. 

Lemma \ref{lem:kalternating2} allows the alternating sets $A_1$ to $A_k$ to be chosen so that each pair of midpoints is at most distance $2k-1$ apart. Therefore, if the random walk has visited at least $2k-1$ different vertices, then it has crossed a midpoint from each $A_i$ and thus is equally likely to end on a red vertex as a blue vertex. It suffices, for example, for the walk to have moved at least distance $2k-1$ from its original position.
\end{proof}

Theorem \ref{the:redpoints} gives a probabilistic condition for the colour to have mixed. The time taken to achieve this condition is as follows.

\begin{Proposition}
After time $\left(8+c\frac{8}{\sqrt{3}}\right)k^2$, the separation distance of the colour distribution from uniform is at most $\frac{1}{c^2}$.
\end{Proposition}
\begin{proof}
For an upper bound on this strong stationary time, recall that the expected time for this simple random walk to move at least $2k-1$ steps from the origin is $2(2k-1)^2$ steps, by for example the analysis of the Gambler's Ruin problem in Proposition $2.1$ of \cite{LPW}, with a factor of two to account for the laziness of the present walk. This is at most $8k^2$ steps. The variance of this time is $\frac{4}{3}((2k-1)^4-(2k-1)^2)$ (see \cite{GamblerMoments}), which is at most $\frac{64}{3}k^4$, so the standard deviation is at most $\frac{8}{\sqrt{3}}k^2$. Chebyshev's inequality completes the proof.
\end{proof}

Thus to apply Theorem \ref{the:redpoints} to any configuration of $n$ red points and $n$ blue points on a cycle, compute the value of $k$ as in Lemma \ref{lem:kalternating1}, and then Theorem \ref{the:redpoints} gives that after time approximately $8k^2$, it is equally likely that the particle is at a red point as at a blue point. After time $\left(8+c\frac{8}{\sqrt{3}}\right)k^2$, the separation distance from uniform is at most $\frac{1}{c^2}$.

\subsection{Colour likelihood}

Theorem \ref{the:redpoints} guarantees that after a certain (stochastic) time, the colour of the present vertex will be uniformly distributed. In general, it is possible that whether red or blue is more likely after $t$ steps can change as $t$ grows --- to see this, imagine the walk starting at a red vertex in a small interval of red vertices, surrounded by ever-thicker bands of alternating colours. If the size of these bands grows rapidly enough, then for each, there will be a range of times $t$ at which it is very likely that the particle is in that band.

In some cases it is possible to say something about whether red or blue is more likely after any number of steps.

\begin{Lemma}
Let $x_0$ be the starting vertex of the random walk. If the closest point of an alternating set $A_i$ to $x_0$ is red, then for any $t$, if the walk is at a point in $A_i$ after $t$ steps, then it is more likely that this point is red than blue.
\end{Lemma}  
\begin{proof}
If the walk crosses a midpoint of $A_i$, then it is equally likely that it ends at a red point of $A_i$ as at a blue point of $A_i$. If the walk does not cross a midpoint of $A_i$, then it can only end at a red point of $A_i$, not at a blue point of $A_i$.
\end{proof}

\begin{Corollary}
\label{cor:closetoreds}
Let $x_0$ be the starting vertex of the random walk. If the closest point of each alternating set $A_i$ to $x_0$ is red, then the walk is always more likely to be at a red point than at a blue point.
\end{Corollary}

For example the random walk of Example \ref{ex:redpointsmod6}, if started at the point $0$, will always be more likely to be at a red point than a blue point. 

\subsection{Generalisations}

A natural generalisation of this scenario is to consider colouring the vertices with more than two colours. For instance, there might be $kn$ vertices, with $n$ of them coloured each of $k$ different colours. When $k$ is a power of two, similar techniques will give upper bounds on the mixing time. If there are four colours, red and orange and blue and green, then Theorem \ref{the:redpoints} could be applied several times --- once to find a time after which it is equally likely for the walk to be at a point which is red or orange or a point which is blue or green, and then after that applied again to find times for the individual colours within these cases to be equal.

Another generalisation is to consider situations where the numbers of vertices of each colour are unequal. Perhaps $p$ of the vertices are red and $(1-p)$ are blue, for example for $p=\frac{1}{5}$, or $p = 0.49$. When $p$ is a simple rational number, this problem is similar to the previous, where some of the colours have been identified.

This problem could also be considered on a more complicated graph than the cycle. For example, an $n \times n$ torus, or the Cayley graph of $S_n$, with half of the points coloured red or blue. The general technique of partitioning into alternating sets might still work, but these sets would not be as simply described as they were in this case.

\section{Further work}

As discussed at the end of \cite{GWcouplingfeatures}, two natural questions in this area are the following:

\begin{Question}
In this paper and \cite{GWcouplingfeatures}, we have been examining the mixing of statistics on a Markov chain. To what extent is it possible to do this in reverse? If one understands the mixing of a `sufficient' set of statistics, what can one conclude about the mixing of the full chain?
\end{Question}

\begin{Question}
Some of the statistics in Section \ref{sec:stats} mix faster than the whole Markov chain, and some do not. Even among those which mix faster than the chain, sometimes it is possible to see this with a relatively simple coupling or strong stationary time, and sometimes note. What controls this behaviour? Which statistics mix faster than the whole chain, and for which statistics are these techniques a good fit?
\end{Question}

\bibliographystyle{plain}
\bibliography{bib}
\end{document}